\documentclass[11pt]{amsart}

\usepackage{latexsym}
\usepackage{amsmath}
\usepackage[active]{srcltx}
\usepackage{amsfonts}
\usepackage{amssymb}

\newtheorem{theorem}{Theorem}[section]
\newtheorem{lemma}[theorem]{Lemma}
\newtheorem{proposition}[theorem]{Proposition}

\newtheorem{example}[theorem]{Example}

\newtheorem{corollary}[theorem]{Corollary}
\newtheorem{remark}[theorem]{Remark}

\newtheorem*{questions*}{Questions}

\def\bh{\mathcal{B(H)}}

\newcommand\Ref{\mathop{\rm Ref}}
\def\gd{\delta}
\def\gs{\sigma}
\def\ga{\alpha}
\def\gl{\lambda}

\def\ot{\otimes}

\newcommand{\prend}{ $\quad\Box$\hfill \bigskip}
\newcommand\Alg{\mathop{\rm Alg}}
\newcommand\Lat{\mathop{\rm Lat}}

\newcommand\nph{\varphi}

\newcommand{\bb}[1]{\mathbb{#1}}
\newcommand{\cl}[1]{\mathcal{#1}}
\newcommand{\sca}[1]{\left\langle#1\right\rangle}
\newcommand{\nor}[1]{\left\Vert #1\right\Vert}

\newcommand{\mat}[3]{\left[\begin{matrix}
1 & #1 & #3 \\
0 & 1 & #2 \\
0 & 0 & 1
\end{matrix}\right]}

\title{Operator algebras from the discrete Heisenberg semigroup}
\author{M.Anoussis, A.Katavolos \& I.G.Todorov}

\begin{document}


\begin{abstract}
We study reflexivity and structure properties of operator algebras
generated by representations of the discrete Heisenberg semi-group.
We show that the left regular representation of this semi-group
gives rise to a semi-simple reflexive algebra. We exhibit an example
of a representation which gives rise to a non-reflexive algebra. En
route, we establish reflexivity results for subspaces of
$H^{\infty}(\bb{T})\otimes\cl B(\cl H)$.
\end{abstract}

\maketitle

\section{Introduction}

The  theory of group
representations has been a motivating force for Operator Algebra
Theory since the very beginnings of the subject. If $\pi$ is a
unitary representation of a group $G$, a much studied object is the
weak-* closed algebra generated by $\{\pi(g):g\in G\}$. A special
case of particular importance arises when $\pi$ is the left regular
representation $g\rightarrow L_g$ acting on $L^2(G)$; the algebra
obtained in this way is the von Neumann algebra $\mathop{\rm VN}(G)$
of the group $G$.

These algebras are all selfadjoint. If $S\subseteq G$ is a
semigroup, one can consider instead the non-selfadjoint algebra
generated by $\{\pi(g):g\in S\}$, possibly restricted to a common
invariant subspace. The algebra of analytic Toeplitz operators is an
instance of this construction. Such algebras have recently attracted
considerable attention in the literature.

Let $\bb{F}_n^+$ be the free semigroup on $n$ generators. The
``non-commutative Toeplitz algebra'' is the weakly closed
algebra generated by the operators $L_g$, $g\in \bb{F}_n^+$,
restricted to the invariant subspace $\ell^2(\bb{F}_n^+)$. It was introduced by Popescu in
\cite{pope} and studied by him in a series of papers and by Arias--Popescu in \cite{ap}.
Later, Davidson--Pitts \cite{davpit2, davpit} and Davidson--Katsoulis--Pitts
 \cite{dkp}, considered this algebra within the more general framework of Free Semigroup Algebras.
 On the other hand,
non-selfadjoint algebras arising from representations of some Lie
groups such as the Heisenberg group, the ``$ax+b$ group'' and
$SL_2(\bb R)$ were considered by Katavolos--Power \cite{hyp,fb}, by
Levene \cite{lev} and by Levene--Power \cite{levpow}. These authors
studied questions including reflexivity and hyperreflexivity,
determination of the invariant subspace lattice and semisimplicity.

In this paper, we study operator algebras arising from
representations of the discrete Heisenberg semigroup. Recall that
the discrete Heisenberg group $\bb H$ consists of all matrices of
the form
$$ \left[ \begin{matrix}
1 & k & n \\
0 & 1 & m \\
0 & 0 & 1
\end{matrix} \right] \qquad k, m, n \in \bb Z.$$
Let $\bb{H}^+$ be the semigroup consisting of all matrices in
$\bb{H}$ with $k,m\in \bb{Z}^+$. We are interested in the weak-*
closed algebra $\cl T_L(\bb{H}^+)$ generated by the operators $L_g$,
$g\in \bb{H}^+$, restricted to the invariant subspace
$\ell^2(\bb{H}^+)$. In Section 4, we show that $\cl T_L(\bb{H}^+)$
contains no non-trivial quasinilpotent or compact elements; in
particular, it is semisimple. We show that the commutant of $\cl T_L(\bb{H}^+)$ is
the corresponding right regular representation and we identify the centre and  the
diagonal. In Section 5 we
prove that $\cl T_L(\bb{H}^+)$ is reflexive using a direct integral
decomposition and the results of Section 3.

In Section 6 we study a class of representations of
$\bb{H}^+$ which 
arise from representations of the irrational rotation algebra
studied by Brenken \cite{br}. The latter, in the multiplicity free
case, are parametrised by a cocycle and a measure. When the cocycle
is trivial, we show that the weak-* closed algebras generated by the
restriction 
to $\bb{H}^+$ are unitarily equivalent to nest algebras or equal to
$\bh$. We also exhibit a representation (corresponding to a
non-trivial cocycle) which generates a non-reflexive algebra even
for the weak operator topology.

In Sections 2 and 3 of the paper we develop a technique that allows
us to handle the question of reflexivity of $\cl T_L(\bb{H}^+)$. We
introduce and study a notion of reflexivity for spaces of operators
acting on tensor products of Hilbert spaces, which we think is of
independent interest. Using this notion, we generalise previous
results of Kraus \cite{kraus} and Ptak \cite{ptak}, establishing
reflexivity for a class of subspaces of $\cl T\otimes\mathcal \bh$
(where $\cl T$ is the algebra of analytic Toeplitz operators). 

\bigskip

\noindent\textbf{Preliminaries and Notation }
The discrete Heisenberg group $\bb{H}$ is 
generated by $$u=\mat{1}{0}{0},\qquad v=\mat{0}{1}{0}\qquad
\text{and }\qquad w=\mat{0}{0}{1}.$$ The element $w$ is central and $uv=wvu$.

We write  $\cl B(\cl H)$ for the algebra
of all bounded linear operators on a Hilbert space $\cl H$.
If $P\in\cl B(\cl H)$ is an (orthogonal) projection, we set $P^{\perp} = I - P$,
where $I$ is the identity operator. We denote by
$\cl B(\cl H)_*$ the predual of $\cl B(\cl H)$, that is, the space
of all weak-* continuous functionals on $\cl B(\cl H)$. If $x,y\in
\cl H$, we write $\omega_{x,y}$ for the vector functional in $\cl
B(\cl H)_*$ given by $\omega_{x,y}(A) = \sca{Ax,y}$, $A\in \cl B(\cl
H)$. If $\cl E$ is a subset of a vector space, $[\cl E]$ will stand for the 
linear span of  $\cl E$. 

The \emph{preannahilator} $\cl S_{\perp}$ of a subspace $\cl S\subseteq \cl B(\cl H)$ is
$$\cl S_{\perp} = \{\omega\in \cl B(\cl H)_* : \omega(A) = 0, \mbox{ for all } \ A\in \cl S\}.$$
The \emph{reflexive hull} of $\cl S$ \cite{ls} is 
$$\Ref\cl S = \{A\in \cl B(\cl H) : \omega_{x,y}(\cl S) = \{0\} \Rightarrow \omega_{x,y}(A) = 0,
\mbox{ for all } x,y\in \cl H\}$$
The subspace $\cl S$ is called \emph{reflexive} if
$\cl S = \Ref\cl S$.

If $\cl L$ is a collection of projections on $\cl H$,
$$\Alg \cl L = \{A\in \cl B(\cl H) : AL = LAL\}$$ is the
algebra of all operators leaving the ranges of the elements of $\cl L$
invariant. It is easy to see that a unital subalgebra $\cl A\subseteq \cl B(\cl H)$ is reflexive
if and only if $\cl A = \Alg\cl L$ for some collection $\cl L$ of projections
on $\cl H$.

Let $\cl H_1$ and $\cl H_2$ be Hilbert spaces and let  $\cl H_1\otimes\cl H_2$ 
be their Hilbert space tensor product. If $\cl S_i\subseteq \cl B(\cl H_i)$, $i = 1,2$,
we let $\cl S_1\otimes\cl S_2$ be the weak-* closed subspace of $\cl B(\cl H_1\otimes\cl H_2)$
generated by the operators $A_1\otimes A_2$, where $A_i \in \cl S_i$, $i = 1,2$.
If $A\in \cl B(\cl H_1)$, we write $A\otimes \cl S_2$ for the space $\bb{C}A \otimes \cl S_2$.
If $\omega_i\in \cl B(\cl H_i)_*$, $i = 1,2$, we let $\omega_1\otimes\omega_2\in \cl B(\cl H_1\otimes\cl H_2)_*$
be the unique weak-* continuous functional satisfying
$(\omega_1\otimes\omega_2)(A_1\otimes A_2) = \omega_1(A_1)\omega_2(A_2)$, $A_i\in \cl B(\cl H_i)$,
$i = 1,2$.

Finally, we let $H^p$ be the Hardy space
corresponding to $p$ ($p = 2,\infty$), that is, the space
consisting by all functions in $L^p(\bb{T})$
whose Fourier coefficients indexed by negative integers vanish.
For each $\nph\in H^{\infty}$, we let
$T_{\nph}\in \cl B(H^2)$ be the analytic Toeplitz operator with symbol $\nph$,
that is, the operator given by
$T_{\nph}f = \nph f$, $f\in H^2$. We let
$$\cl T =\{T_{\nph} : \nph\in H^{\infty}\}$$ be the algebra of all
analytic Toeplitz operators on $H^2$.

\section{A reflexive hull for subspaces of $\cl B(\cl H_1\otimes\cl H_2)$}\label{s_rhs}

In this section, we introduce a reflexive hull for
spaces of operators that act on the tensor product of two given Hilbert spaces.
The results will be applied in Section 3 to
study reflexivity of subspaces of $\cl T\ot \cl B(\cl K)$ for a given Hilbert space $\cl K$.

Suppose a Hilbert space $\cl H$ decomposes as a tensor product $\cl H_1\ot\cl H_2$ of two Hilbert spaces.
If $\omega\in \cl B(\cl H_1)_*$ then the
right slice map $R_{\omega} : \cl B(\cl H_1\otimes\cl
H_2)\rightarrow \cl B(\cl H_2)$ is the unique
weak-* continuous linear map with the property that
$R_{\omega}(A\otimes B) = \omega(A)B$, whenever $A\in \cl B(\cl H_1)$
and $B\in \cl B(\cl H_2)$. Similarly one defines the left slice maps, denoted by $L_{\tau}$,
where $\tau\in \cl B(\cl H_2)_*$.
We note that if $\omega = \omega_{\xi,\eta}$ for some vectors
$\xi,\eta\in \cl H_1$ then for all $x,y\in\cl H_2$,
\begin{equation}\label{red}
\sca{R_{\omega}(T)x,y} = \sca{T(\xi\otimes x),\eta\otimes y}, \ \ \ T\in \cl B(\cl H_1\otimes\cl H_2).
\end{equation}
This equality shows that, when $\omega$ is a vector functional
or, more generally, a weakly continuous functional,
then $R_\omega$ is also weakly (that is, WOT-WOT) continuous.

If $\cl S$ is a weak-* closed subspace of $\cl B(\cl H_1)$ and $T\in\cl S\ot\cl{B(H}_2)$,
then clearly $L_\omega(T)\in\cl S$ for all $\omega\in \cl{B(H}_2)_*$. The converse was proved in \cite{kraus}:

\begin{lemma}[Kraus]\label{slice}
Let $\cl S$ be a weak-* closed subspace of $\cl B(\cl H_1)$ and $T\in\cl{B(H}_1\ot\cl H_2)$.
If $L_\omega(T)\in\cl S$ for all $\omega\in \cl{B(H}_2)_*$ then $T\in\cl S\ot\cl{B(H}_2)$.
\end{lemma}
Consider the  set of vector functionals
$$\cl E = \{\omega_{\xi\otimes x,\eta\otimes y} :
\xi,\eta\in \cl H_1, x,y\in \cl H_2\}\subseteq \cl B(\cl H_1\otimes\cl H_2)_*.$$
The set $\cl E$ (as any subset of the dual of $\cl B(\cl H_1\otimes\cl H_2)$, see
\cite{hadwin}) can be used to define
a reflexive hull for subspaces of $\cl B(\cl H_1\otimes\cl H_2)$. Namely,
if $\cl S\subseteq\cl B(\cl H_1\otimes\cl H_2)$ let
$$\Ref\mbox{}_e\cl S = \{T\in \cl B(\cl H_1\otimes\cl H_2) : \omega(\cl S) = \{0\} \Rightarrow \omega(T) = 0,
\ \ \forall \ \omega\in \cl E\}.$$

It is clear that $\Ref_e(\cl S)$ depends on the tensor product decomposition $\cl{H=H}_1\otimes\cl H_2$.
The following statements are easy consequences of the definition; we omit their proofs.

\begin{lemma}\label{easin}
Let $\cl S\subseteq\cl B(\cl H_1\otimes\cl H_2)$. Then

(i) \ \ $\Ref_e\cl S$ is a reflexive, hence weakly closed, subspace of operators;

(ii) \ $\Ref\cl S\subseteq\Ref_e\cl S$;

(iii) $\Ref_e\cl S = \Ref_e\Ref\cl S = \Ref_e\Ref_e\cl S$.
\end{lemma}

It follows from Lemma \ref{easin} that if a subspace $\cl
S\subseteq\cl B(\cl H_1\otimes\cl H_2)$ satisfies $\Ref_e\cl S=\cl
S$ then $\cl S$ is reflexive. Remark \ref{bigger} below shows that
the converse does not hold.

\begin{lemma}\label{fubini}
Let $\cl U\subseteq  \cl B(\cl H_1)$ and $\cl V\subseteq \cl B(\cl H_2)$ be subspaces. Then
\[
\Ref\mbox{}_e (\cl U\ot\cl V) = ( \cl B(\cl H_1)\ot\Ref\cl V)\cap ( \Ref\cl U\ot B(\cl H_2)).
\]
\end{lemma}
\proof Note that a vector functional
$\omega_{\xi\otimes x, \eta\otimes y} = \omega_{\xi,\eta}\ot \omega_{x,y}$ annihilates $\cl U\ot\cl V$
if and only if either $\omega_{\xi,\eta}$ annihilates $\cl U$ or $\omega_{x,y}$ annihilates $\cl V$.
For if there exists $U\in \cl U$ with $\omega_{\xi,\eta}(U)\ne 0$, then for all $V\in\cl V$ we have
$\omega_{\xi,\eta}(U) \omega_{x,y}(V)=0$, and hence $\omega_{x,y}(V)=0$.

Now let $T\in \Ref\mbox{}_e (\cl U\ot\cl V)$. Suppose that
$\omega_{\xi,\eta}\in \cl U_{\perp}$. Then $\omega_{\xi\otimes
x,\eta\otimes y}$ annihilates $\cl U\ot\cl V$ for all $x,y\in \cl
H_2$, and hence
\[
\omega_{\xi,\eta}(L_{\omega_{x,y}}(T))=(\omega_{\xi,\eta}\ot
\omega_{x,y})(T) = 0.
\]
This shows that $L_{\omega_{x,y}}(T)\in\Ref\cl U$. Since $x,y\in\cl H_2$ are arbitrary,
linearity and (norm) continuity of the map $\omega\to L_\omega$ yields $L_\omega(T)\in\Ref\cl U$ for all
$\omega\in (\cl{B(H}_2))_*$. By Lemma \ref{slice}, $T\in \Ref\cl U\ot\cl{B(H}_2)$.
Similarly, one obtains $T\in\cl{B(H}_1)\ot \Ref\cl V$.

Conversely, if $T\in ( \cl B(\cl H_1)\ot\Ref\cl V)\cap ( \Ref\cl U\ot \cl B(\cl H_2))$, then for each
$\phi=\omega_{\xi,\eta}\, (\xi, \eta\in\cl{H}_1)$ we have $R_{\phi}(T)\in\Ref\cl V$.
So if $\omega_{x,y}$ is a vector functional
annihilating $\cl V$ then it must annihilate $ R_{\phi}(T)$, and hence
\[
\phi(L_{\omega_{x,y}}(T))=(\phi\ot\omega_{x,y})(T)=\omega_{x,y}(R_{\phi}(T))=0.
\]
Since $\phi=\omega_{\xi,\eta}$ with $\xi, \eta$ arbitrary in  $\cl{H}_1$, this implies $L_{\omega_{x,y}}(T)=0$.
Similarly, using the fact that all left slices of $T$ must lie in $\Ref\cl U$,
we see that
\[
\omega_{\xi,\eta}\in\cl U_\perp\;\;\Rightarrow\;\; R_{\omega_{\xi,\eta}}(T)=0.
\]
Therefore, if $\omega_{\xi,\eta}\ot \omega_{x,y}$ annihilates $\cl U\ot\cl V$ then
either $\omega_{\xi,\eta}$ annihilates $\cl U$, in which case $R_{\omega_{\xi,\eta}}(T)=0$,
or $\omega_{x,y}$ annihilates $\cl V$, in which case $L_{\omega_{x,y}}(T)=0$.
In either case,
\[
(\omega_{\xi,\eta}\ot \omega_{x,y})(T)=\omega_{\xi,\eta}(L_{\omega_{x,y}}(T))
=\omega_{x,y}(R_{\omega_{\xi,\eta}}(T))=0
\]
which shows that $T\in \Ref\mbox{}_e (\cl U\ot\cl V)$.
 \prend

\begin{remark}
{\rm The intersection $(\cl B(\cl H_1)\ot\cl V)\cap (\cl U\ot \cl
B(\cl H_2))$ coincides with the Fubini product $F(\cl U, \cl V)$
defined by Tomiyama in \cite{to2} for von Neumann algebras and by
Kraus in \cite{kraus} for weak-* closed spaces of operators.

Let $\cl L_1$ and $\cl L_2$ be subspace lattices on the Hilbert
spaces $\cl H_1,\cl H_2$ and $\cl L_1\ot\cl L_2$ be the smallest
subspace lattice generated by $P_1\ot P_2$, where $P_i\in\cl L_i$,
$i = 1,2$. It follows from a result  of Kraus \cite[Eq.
(3.3)]{kraus} that the Fubini product $ F(\Alg\cl L_1, \Alg\cl L_2)$
equals $\Alg(\cl L_1\ot\cl L_2)$. Combining this with Lemma
\ref{fubini} we obtain
$$\Ref\mbox{}_e (\Alg\cl L_1 \ot \Alg\cl L_2)=\Alg(\cl L_1\ot\cl L_2).$$}
\end{remark}

\begin{corollary}\label{new}
(a) If  $A\in\cl B(\cl H_1)$, then
$\Ref_e (A\otimes\cl V) = A\otimes \Ref\cl V$.

(b) If $\cl U \subseteq \cl B(\cl H_1)$ then
$\Ref_e (\cl U\ot\cl{B(H}_2)) = \Ref\cl U\ot\cl{B(H}_2)$.
\end{corollary}
\proof
(a) Clearly, we may assume that $A\neq 0$.
If $T\in\Ref_e (A\otimes\cl V)$, then by Lemma \ref{fubini},
$T\in ( \cl B(\cl H_1)\ot\Ref\cl V)\cap (\Ref \bb C A \ot\cl B(\cl H_2))$. But $\Ref \bb C A=\bb C A$
since one dimensional subspaces are reflexive (see \cite[56.5]{con}, for example),
so $T=A\ot B$ for some $B\in\cl B(\cl H_2)$.
Thus $A\ot B\in\cl B(\cl H_1)\ot\Ref\cl V$, which implies that $B\in\Ref\cl V$.

(b) follows from Lemma \ref{fubini}.
\prend

\begin{lemma}\label{slic}
Let $\cl S\subseteq\cl B(\cl H_1\otimes \cl H_2)$ be a subspace of operators
and $\omega \in \cl B(\cl H_1)_*$ be a vector functional.
Then $R_{\omega}(\Ref_e \cl S)\subseteq \Ref R_{\omega}(\cl S)$.

Similarly, if $\tau\in \cl B(\cl H_2)_*$ is a vector functional
then $L_{\tau}(\Ref_e \cl S)\subseteq \Ref L_{\tau}(\cl S)$.
\end{lemma}
\proof Let $\omega = \omega_{\xi,\eta}$, where
$\xi,\eta\in \cl H_1$. 
Fix $T\in \Ref_e \cl S$ and suppose that $x,y\in \cl H_2$
are such that $\omega_{x,y}(R_{\omega}(\cl S)) = \{0\}$.
It follows from (\ref{red}) that
$$\omega_{\xi\otimes x, \eta\otimes y}(\cl S) = \{0\}.$$
Since $T\in \Ref_e\cl S$, we have that
$\omega_{\xi\otimes x, \eta\otimes y}(T) = \{0\}$. By (\ref{red}) again,
$\omega_{x,y}(R_{\omega}(T))$ $= \{0\}$. We showed that $R_{\omega}(T)\in \Ref R_{\omega}(\cl S)$.
The first claim is proved.
The second claim follows similarly.
\prend

\begin{proposition}\label{distinct}
For a projection $L\in \cl B(\cl H_1\otimes\cl H_2)$, let
$\tilde{L}$ be the projection onto the subspace
$\{\xi \otimes x : L(\xi\otimes x) = 0\}^{\perp}$.
Let $P,Q\in \cl B(\cl H_1\otimes\cl H_2)$ be projections.
Then
$$\Ref\mbox{}_e Q\cl B(\cl H_1\otimes\cl H_2)P =  \tilde{Q}\cl B(\cl H_1\otimes\cl H_2)\tilde{P}.$$

In particular,
there exists a subspace $\cl S\subseteq\cl B(\cl H_1\otimes\cl H_2)$
such that $\Ref_e\cl S$ is strictly bigger than $\Ref\cl S$.
\end{proposition}
\proof
Fix projections $P,Q\in \cl B(\cl H_1\otimes\cl H_2)$ and let
$\cl S = Q\cl B(\cl H_1\otimes\cl H_2)P$.
It is clear that
\begin{equation}\label{eq_anni}
\cl S_{\perp}\cap \cl E = \{\omega_{\xi\otimes x, \eta\otimes y} : P(\xi\otimes x) = 0 \mbox{ or }
Q(\eta\otimes y) = 0\}.
\end{equation}
Hence, $T\in \Ref_e\cl S$ if and only if
$\sca{T(\xi\otimes x),\eta\otimes y} = 0$ for all $\xi,\eta\in \cl H_1$ and all $x,y\in \cl H_2$
such that either $P(\xi\otimes x) = 0$ or $Q(\eta\otimes y) = 0$.

Suppose $T\in \Ref_e\cl S$.
If $\xi\in \cl H_1$ and $x\in \cl H_2$ are such that
$P(\xi\ot x)=0$ then for any $\eta\in \cl H_1$ and $y\in \cl H_2$ we have $\sca{T(\xi\otimes x),\eta\otimes y} = 0$
and so $T(\xi\otimes x) = 0$. But $\tilde P^\bot(\cl H_1\otimes\cl H_2) = \overline{[\xi\ot x: P(\xi\ot x)=0]}$.
It follows that $T\tilde P^\bot = 0$, or $T=T\tilde P$.
By considering adjoints, we conclude that $T=\tilde QT$, and
thus $T = \tilde{Q}T\tilde{P}$.
Conversely, if $T$ is of this form then $T\in \Ref_e\cl S$ by the previous paragraph.

For the last statement, it is enough to exhibit a projection $P\in
\cl B(\cl H_1\otimes\cl H_2)$ such that $\tilde{P}$ is strictly
greater than $P$. It suffices to choose any $P \neq I$ which
annihilates no non-trivial elementary tensors. For example, take
$P=F^\bot$, where $F$ is the projection onto $\{\lambda(e_1\otimes
f_1 + e_2\otimes f_2) : \lambda\in \bb{C}\}$ and where
$\{e_1,e_2\}\subseteq\cl H_1$ (resp. $\{f_1,f_2\}\subseteq\cl H_2$)
is linearly independent. Here 
$P\ne I$ but $\tilde P=I$. \prend

\begin{example}\label{bigger}
Let $\cl H_1$ be infinite dimensional, $V\in\cl{B(H}_1)$ be an
isometry and $\cl S\subseteq \cl B(\cl H_2)$ be a weak-* closed
subspace. Then

$(i)$ \ $\Ref (V\ot\cl S) =V\ot\cl S$.

$(ii)$ If $\cl S$ is not reflexive, then
\[ \Ref (V\ot\cl S)\subsetneqq\Ref\mbox{}_e (V\otimes\cl S).
\]
\end{example}
\proof The  equality $\Ref (V\ot\cl S) =V\ot\cl S$ is well-known when $V$ is the identity
(see \cite[59.7]{con}, for example),
and the proof readily extends to
the general case.

Since $\Ref_e (V\ot\cl S)=V\ot\Ref\cl S$ by Corollary \ref{new},
 if $\cl S$ is not reflexive, then $\Ref (V\otimes\cl S)$
is strictly contained in $\Ref_e (V\ot\cl S)$.
\prend

\section{Reflexive hulls and Fourier coefficients}\label{s_rhfc}

We recall that for each $\nph\in H^{\infty}$, we denote by
$T_\nph$ the analytic Toeplitz operator on $H^2$ with symbol $\nph$ and by
$\cl T$ the collection of all analytic Toeplitz operators on $H^2$.
Let $\zeta_n\in H^2$ be the function given by $\zeta_n(z)=z^n$,  $z\in\bb T$.
We note that $\{\zeta_{n} : n\geq 0\}$ is an orthonormal basis of $H^2$.
Let $S=T_{\zeta_1}\in \cl T$ be the unilateral shift.

For the rest of this section, we fix a Hilbert space $\cl K$.
We note that $(\cl T\otimes\cl B(\cl K))' = \cl T\otimes I$.
Indeed, if $T\in (\cl T\otimes\cl B(\cl K))'$ then $T\in (I\otimes\cl B(\cl K))'$ and hence
$T = A\otimes I$ for some $A\in \cl B(H^2)$.
It now follows that $A\in \cl T' = \cl T$ \cite{sara}.
Thus, $(\cl T\otimes\cl B(\cl K))'' = (\cl T\otimes I)'$.
Now, if $X\in (\cl T\otimes I)'$ then $X(T\otimes I) = (T\otimes I)X$
for all $T\in \cl T$. Applying left slice maps we obtain $L_{\omega}(X)T = TL_{\omega}(X)$
for all normal functionals $\omega$ and all $T\in \cl T$. Thus, $L_{\omega}(X)\in \cl T' = \cl T$ for
all normal functionals $\omega$, which means by Lemma \ref{slice} that
$X\in \cl T\otimes \cl B(\cl K)$.
We conclude that $(\cl T\otimes\cl B(\cl K))''=\cl T\otimes\cl B(\cl K)$
and in particular that $\cl T\otimes\cl B(\cl K)$ is automatically weakly closed.

If $T\in \cl T\otimes\cl B(\cl K)$, let $\hat{T}_{n}$, where $n \geq
0$, be the operators determined by the identity
$$T(\zeta_{0}\otimes x) = \sum_{n\geq 0} \zeta_{n}\otimes \hat{T}_n x, \ \ \ x\in \cl K.$$
Alternatively, $\hat{T}_n = R_{\omega_n}(T)$, where $\omega_n=\omega_{\zeta_0,\zeta_n}$, $n\geq 0$.

We call $\sum_{n \geq 0} S^{n}\otimes \hat{T}_{n}$
the formal Fourier series of $T$.
When $\cl K$ is one dimensional, this is the usual Fourier series of
an operator $T\in \cl T$.
By standard arguments, as in the scalar case, the Cesaro sums of this series
converge to $T$ in the weak-* topology.

If $\bb{S}$ is a family $(\cl S_n)_{n\geq 0}$ of subspaces of
$\cl B(\cl K)$, we let
$$\cl A(\bb{S}) = \{T\in \cl T\otimes \cl B(\cl K) : \hat{T}_n\in \cl S_n, \ \ n\geq 0\}.$$
It is obvious that $\cl A(\bb{S})$ is a linear space; it is a subalgebra of $\cl B(H^2\otimes \cl K)$
if and only if $\cl S_n\cl S_m\subseteq \cl S_{n+m}$, for all $n,m\geq 0$.

\begin{remark}\label{wcl}
If $\cl S_n$ is closed in the weak operator (resp. the weak-*)
topology and $\bb{S} = (\cl S_n)_{n\geq 0}$ then $\cl A(\bb{S})$ is
closed in the weak operator (resp. the weak-*) topology.
\end{remark}

This follows from the fact that the slice maps $R_{\omega_n}$ are continuous
both in the weak-weak and the weak-*-weak-* sense.

\begin{remark}\label{tens}
If $\cl S \subseteq \cl B(\cl K)$ is a weak-* closed space and $\cl S_n = \cl S$ for each $n\geq 0$,
then
$\cl A(\bb{S}) = \cl T\otimes\cl S$.
\end{remark}
Indeed, if $A\in \cl S$ and $k\geq 0$ then obviously $S^k\otimes A\in \cl A(\bb{S})$
and hence  $\cl T\otimes\cl S\subseteq\cl A(\bb{S})$ since the latter is weak-* closed.

Conversely, suppose that $T\in \cl T\otimes\cl B(\cl K)$ is such that
$\hat{T}_n\in \cl S$ for each $n\geq 0$. Then $S^n\otimes\hat{T}_n\in \cl T\otimes\cl S$
and hence the Cesaro sums of the Fourier series of $T$ are in   $\cl T\otimes\cl S$.
But $\cl T\ot\cl S$ is weak-* closed, and so $T\in \cl T\otimes\cl S$.

\medskip

If $\bb{S} = (\cl S_n)_{n\geq 0}$ we let
$\Ref\bb{S} \stackrel{def}{=} (\Ref\cl S_n)_{n \geq 0}$.

\begin{theorem}\label{sref}
If $\bb{S} = (\cl S_n)_{n\geq 0}$ is a sequence of subspaces of $\cl B(\cl K)$ then 
$\Ref_e \cl A(\bb{S}) = \cl A(\Ref\bb{S})$.
In particular, if $\cl S_n$ is reflexive for each
$n\geq 0$ then $\cl A(\bb{S})$ is reflexive.
\end{theorem}
\proof
First observe that $\Ref_e(\cl T\otimes\cl B(\cl K)) = (\Ref\cl T)\otimes\cl B(\cl K)$ 
by Corollary \ref{new}. But $\cl T$ is reflexive \cite{sara}
and hence $\Ref_e(\cl T\ot\cl B(\cl K)) = \cl T\ot\cl B(\cl K)$.

Let $T\in \Ref_e\cl A(\bb{S})$. As just observed, $T\in \cl T\otimes\cl B(\cl K)$.
By Lemma \ref{slic}, for each $n\ge 0$, writing $\omega_n=\omega_{\zeta_0,\zeta_n}$, we have
$$R_{\omega_n}(T)\in \Ref R_{\omega_n}(\cl A(\bb{S}))\subseteq \Ref\cl S_n$$
since  $R_{\omega_n}(\cl A(\bb{S}))\subseteq \cl S_n$ by the definition of $\cl A(\bb S)$.
In other words, $\hat{T}_n\in \Ref\cl S_n$ for all  $n\geq 0$, and so $T\in \cl A(\Ref\bb{S})$.

Conversely, suppose that $T\in \cl A(\Ref\bb{S})$, that is, $\hat{T}_n\in \Ref\cl S_n$ for
each $n\geq 0$. By Corollary \ref{new},
$S^n\otimes \hat{T}_n\in \Ref_e(S^n\otimes \cl S_n)$,
$n\geq 0$. Since $S^n\otimes\cl S_n\subseteq\cl A(\bb{S})$, we conclude that
$S^n\otimes \hat{T}_n\in \Ref_e \cl A(\bb{S})$, $n\geq 0$.
By Lemma \ref{easin} (i) and the fact that $T$ is in the weak-* closed linear hull of
$\{S^n\otimes \hat{T}_n : n\geq 0\}$ we have that $T\in \Ref_e\cl A(\bb{S})$.

Suppose that $\cl S_n$ is reflexive for each $n\geq 0$. By Lemma \ref{easin} (ii)
and the first part of the proof,
$$\cl A(\bb{S}) \subseteq\Ref\cl A(\bb{S})\subseteq \Ref\mbox{}_e\cl A(\bb{S}) = \cl A(\bb{S})$$
and hence $\cl A(\bb{S})$ is reflexive.
\prend

As an immediate corollary of Theorem \ref{sref}  we obtain the following result,
proved for reflexive algebras by Kraus \cite{kraus} and Ptak \cite{ptak}.

\begin{corollary}\label{ptak}
Let $\cl S\subseteq \cl B(\cl K)$ be a reflexive subspace. Then $\cl T\otimes\cl S$
is reflexive.
\end{corollary}

\begin{remark}
{\rm We note that $\Ref_e\cl A(\bb{S})$ is in general
strictly larger than $\Ref\cl A(\bb{S})$.
Indeed, let $\cl S\subseteq \cl B(\cl K)$ be a non-reflexive weak-* closed subspace
and $\bb{S} = (\cl S_n)_{n\geq 0}$ be the family with $\cl S_1 = \cl S$
and $\cl S_n = \{0\}$ if $n\neq 1$. Then
$\cl A(\bb{S}) = S\otimes\cl S$ is reflexive (Example \ref{bigger} (i)).
However, by Theorem \ref{sref},
$\Ref_e\cl A(\bb{S}) = S\otimes \Ref\cl S$ which strictly contains
$\cl A(\bb{S})$.}
\end{remark}


The following corollary will be used in 
Theorem \ref{drinfol}.

\begin{corollary}\label{hase}
Let $U,V\in \cl B(\cl K)$ satisfy $UV = \lambda VU$ for some $\lambda\in\bb{C}$.
Suppose that $V$ is invertible and that the weak-* closure $\cl W_0$
of the polynomials in $U$ is reflexive.
Then the weak-* closed unital operator algebra $\cl W\subseteq\cl B(H^2\otimes \cl K)$
generated by $I\otimes U$ and $S\otimes V$ is reflexive.
\end{corollary}
\proof
The commutation relation $UV = \lambda VU$ implies that $\cl W$ is the weak-*
closed linear hull of the set $\{S^k\otimes V^kU^m : k,m\geq 0\}$.

Let $\bb{S} = (V^n\cl W_0)_{n\geq 0}$.
We claim that $\cl W = \cl A(\bb{S})$.
Suppose that $T\in \cl T\otimes\cl B(\cl K)$ and that $\hat{T}_n\in V^n\cl W_0$,
$n\geq 0$. Then
$$S^n\otimes \hat{T}_n\in S^n\otimes V^n\cl W_0 = (S^n\otimes V^n)(I\otimes\cl W_0)\subseteq \cl W.$$
It follows by approximation (in the w*-topology) that $T\in \cl W$. Thus, $\cl A(\bb{S})\subseteq\cl W$.

To show that $\cl W\subseteq \cl A(\bb{S})$, it suffices to prove that
$S^k\otimes V^kU^m\in \cl A(\bb{S})$, for each $k,m\geq 0$. So, fix such $k$ and $m$
and note that, if $x,y\in \cl K$ then
\begin{eqnarray*}
\sca{R_{\omega_{\zeta_0,\zeta_n}}(S^k\ot V^kU^m)x,y} & = & \sca{(S^k\ot V^kU^m)(\zeta_0\ot x),\zeta_n\otimes y}\\
& = & \sca{\zeta_k\ot V^kU^m x, \zeta_n\ot y}
=  \delta_{k,n}\sca{V^kU^m x,y}.
\end{eqnarray*}
Thus $R_{\omega_{\zeta_0,\zeta_n}}(S^k\ot V^kU^m) = \delta_{k,n}V^kU^m\in V^n\cl W_0$
for all $n$ and hence $S^k\ot V^kU^m \in \cl A(\bb S)$ as required.

Now observe that, since $V$ is invertible and $\cl W_0$ is reflexive, each $\cl S_n=V^n\cl W_0$ is reflexive.
It therefore follows from Theorem \ref{sref} that $\cl W=\cl A(\bb S)$ is reflexive.
\prend

\begin{remark}
Part of Theorem \ref{sref} and Corollary \ref{ptak}
have been independently obtained by Kakariadis \cite{kaka}. 
\end{remark}

\section{The structure of $\cl T_L(\bb H^+)$}

In this section we study the weak-* closed operator algebra
$\cl T_L(\bb H^+)$ generated by the image of the left
regular representation of $\bb{H}^+$ restricted to the invariant subspace
$\cl H=\ell^2(\bb{H}^+)$.
We identify $\cl H$ with  $\ell^2(\bb Z) \otimes\ell^2(\bb{Z}_+)\otimes \ell^2(\bb{Z}_+)$,
where the element of the canonical orthonormal basis of $\cl H$
corresponding to $w^nu^kv^m\in\bb H^+$
is identified with the elementary tensor $w^n\otimes u^k\otimes v^m$.
Then $\cl T_L(\bb H^+)$ is generated by the operators $L_u,L_v$ and $L_w$ on $\cl H$ which act  as follows:
\begin{align*}
L_u(w^n\otimes u^k\otimes v^m) = &w^n\otimes u^{k+1}\otimes v^m \\
L_v(w^n\otimes u^k\otimes v^m) = &w^{n-k}\otimes u^k\otimes v^{m+1}\\
L_w(w^n\otimes u^k\otimes v^m) = &w^{n+1}\otimes u^{k}\otimes v^m,  \qquad (n,k,m)\in \bb Z\times\bb Z_+ \times\bb Z_+.
\end{align*}
By the commutation
relations, $\cl T_L(\bb{H}_+)$ coincides with the weak-* closed
linear span of the set
$$\{L_w^n L_u^k L_v^m: (n,k,m)\in \bb Z\times\bb Z_+ \times\bb Z_+\}.$$

Throughout this section we will identify $\ell^2(\bb{Z})$ with $L^2(\bb{T})$ via
Fourier transform in the first coordinate $w$.
In this way, the identity function $\zeta_1$ on $\bb T$ is identified with $w$
and $\cl T_L(\bb H^+)$ is
identified with an operator algebra acting on $L^2(\bb
T)\ot\ell^2(\bb Z_+\times\bb Z_+)$.
Let $\cl C$ be the weak-*
closed linear span of $\{L_w^n:n\in\bb Z\}$. This is an abelian von
Neumann algebra; it consists of all operators $\{L_f:f\in
L^\infty(\bb T)\}$ where
\[
L_f(w^n\otimes u^k\otimes v^m) = (fw^n)\otimes u^k\otimes v^m.
\]
Thus $\cl C=\cl M\ot 1\ot 1$, where
$\cl M\subseteq \cl B(L^2(\bb T))$ is the multiplication masa of $L^\infty(\bb T)$.

If $(e^{is},e^{it})\in \bb{T}\times\bb{T}$ ($s,t\in [0,2\pi)$),  
let $W_{s,t}\in \cl B(\cl H)$ be the
unitary operator given by
$$W_{s,t}(w^n\otimes u^k\ot v^m) = w^n\otimes e^{isk} u^k\otimes e^{itm}v^m,
\ \ \ (n,k,m)\in \bb{Z}\times\bb{Z}_+\times\bb{Z}_+.$$
We define an action of the two-torus $\bb{T}\times \bb{T}$ on $\bh$ by
$$\rho_{s,t}(A) = W_{s,t}AW_{s,t}^*, \ \ A\in \bh.$$
Observe that
\[\rho_{s,t}(L_u) = e^{is}L_u, \qquad  \rho_{s,t}(L_v) = e^{it}L_v, \qquad  \rho_{s,t}(L_w) = L_w.\]
Hence, $\rho_{s,t}$ leaves $\cl T_L(\bb H^+)$ invariant. Since $\rho_{s,t}$
is unitarily implemented, it also leaves $\Ref\cl T_L(\bb H^+)$ invariant.

If $A\in\cl T_L(\bb H^+)$ is a ``trigonometric polynomial'', namely a sum
\[A = \sum_{(k,m)\in \Omega} L_{f_{k,m}}L_u^kL_v^m,\]
where $\Omega\subseteq \bb{Z}_+\times\bb{Z}_+$ is finite and
$f_{k,m}\in L^{\infty}(\bb{T}) \, ((k,m)\in \Omega)$,
then it is easy to observe that
\begin{align*}
\frac{1}{4\pi^2}\int_0^{2\pi}\!\int_0^{2\pi} \rho_{s,t}(A)e^{-isk}e^{-itm}dtds=L_{f_{k,m}}L_u^kL_v^m.
\end{align*}

We will need the following proposition, which is a version of
well-known facts adapted to our setting.

\begin{proposition}\label{tree}
For $k,m\in \bb{Z}_+$, let $Q_{k,m}\in \cl B(\cl H)$ be the
orthogonal projection onto the subspace $L^2(\bb T)\ot [u^k]\ot [v^m]$
spanned by the vectors of the form $f\ot u^{k}\ot v^{m}$, $f\in L^2(\bb T)$.
If $A \in \cl\bh$ and $p,q\in\bb Z$, set
\[\Phi_{p,q}(A)=\sum_{k,m} Q_{k+p,m+q}AQ_{k,m},\]
where the sum is taken over all $k,m \in \bb{Z}_+$ such that $k+p,m+q\in \bb{Z}_+$.
The following statements hold:
\begin{enumerate}
\item
$\displaystyle \Phi_{p,q}(A)=\frac{1}{4\pi^2}\int_0^{2\pi}\!\!\int_0^{2\pi}
\rho_{s,t}(A)e^{-isp}e^{-itq}dtds$.
\item If $0<r<1$ then the series \[\sum_{p,q\in \bb Z} \Phi_{p,q}(A)  r^{|p|+|q|}\]
converges absolutely in norm to an operator $A_r$; moreover,
$\nor{A_r}\le \nor{A}$ and $\text{w*-}\lim_{r\nearrow 1}A_r = A$.
\item
If $\Phi_{p,q}(A) = 0$ for all $p,q\in\bb{Z}$ then $A = 0$.
\item If $A\in\cl T_L(\bb H^+)$ and $B\in\Ref\cl T_L(\bb H^+)$ then
$\Phi_{p,q}(A), A_r\in\cl T_L(\bb H^+)$ and
 $\Phi_{p,q}(B), B_r\in \Ref\cl T_L(\bb H^+)$, for all $p,q\in \bb{Z}$.
\end{enumerate}
\end{proposition}
\proof
{\bf (1)} Let $x = Q_{k_1,m_1}x$ and $y = Q_{k_2,m_2}y$. We have
$$\sca{\Phi_{p,q}(A)x,y}=\sca{\sum Q_{k+p,m+q}AQ_{k,m}x,y}=\delta_{k_1+p,k_2}\delta_{m_1+q,m_2}\sca{Ax,y},$$
where the summation takes place over
all $k,m\in\bb Z_+$ with $k+p, m+q\in \bb{Z}_+$. On
the other hand, we have
$$\sca{\rho_{s,t}(A)x,y}= \sca{W_{s,t}AW_{s,t}^*x,y}=e^{-isk_1-itm_1}e^{isk_2+itm_2}\sca{Ax,y}$$
and hence
\begin{align*}
&\frac{1}{4\pi^2}\int_0^{2\pi}\!\!\int_0^{2\pi} \sca{\rho_{s,t}(A)e^{-isp}e^{-itq}x,y}dtds= \\
&\frac{1}{4\pi^2}\int_0^{2\pi}\!\!\int_0^{2\pi} \sca{Ax,y}e^{is(k_2-k_1-p)}e^{it(m_2-m_1-q)}dtds
=\delta_{k_1+p,k_2}\delta_{m_1+q,m_2}\sca{Ax,y}.
\end{align*}
\medskip
{\bf (2)}
Let $F$ be the operator valued function defined on $\bb{T}\times\bb{T}$ by
$F(s,t)=\rho_{s,t}(A)$, and let $\hat{F}$ be its Fourier transform. By (1),
$\hat F(p,q)=\Phi_{p,q}(A)$. If $P_r(s,t)$
denotes the two-dimensional Poisson kernel, then one readily sees that $A_r=(F*P_r)(0,0)$.

The claim therefore follows from the well-known properties of the Poisson kernel.

\smallskip
{\bf (3)} is an immediate consequence of (2).

\smallskip
{\bf (4)}
It follows from (1) that $\Phi_{p,q}(A) \in\cl T_L(\bb H^+)$ and
$\Phi_{p,q}(B) \in \Ref\cl T_L(\bb H^+)$, since $\rho_{s,t}$ leaves
$\cl T_L(\bb H^+)$ and  $\Ref\cl T_L(\bb H^+)$ invariant.
Now (2) implies that $A_r\in\cl T_L(\bb H^+)$ and
$B_r\in \Ref\cl T_L(\bb H^+)$.
\endproof

We isolate some  consequences of Proposition \ref{tree}: 

\begin{corollary}\label{five}
If $A\in\cl T_L(\bb H^+)$ then \\ (a) $\Phi_{k,m}(A)=0$ unless $k\ge 0$ and $m\ge 0$. \\
(b) For each $k,m\ge 0$, the operator $L_{k,m}\equiv (L_v^m)^*(L_u^k)^*\Phi_{k,m}(A)$
is in $\cl C$. Hence there exists $f_{k,m}(A)\in L^\infty(\bb T)$ such that
$L_{k,m}=L_{f_{k,m}(A)}$. We have
$\Phi_{k,m}(A)=L_{f_{k,m}}(A)L_u^kL_v^m$.
\end{corollary}
\begin{proof}
Since $\cl T_L(\bb H^+)$
is the weak-* closed hull of its trigonometric polynomials and
the map $\Phi_{k,m}$ is weak-* continuous, it suffices to assume that $A$ is of the form
$A = \sum_{(k,m)\in \Omega} L_{f_{k,m}}L_u^kL_v^m$,
where $\Omega\subseteq \bb{Z}_+\times\bb{Z}_+$
is finite. Now (a) is obvious. For (b), we have
\[
\Phi_{k,m}(A)=\frac{1}{4\pi^2}\int_0^{2\pi}\!\int_0^{2\pi} \rho_{s,t}(A)e^{-isk}e^{-itm}dtds
=L_{f_{k,m}}L_u^kL_v^m= L_u^kL_v^mL_{f_{k,m}},
\]
hence 
$ (L_v^m)^*(L_u^k)^*\Phi_{k,m}(A) = L_{f_{k,m}}$ which is in $\cl C.$
\end{proof}

We can now identify the diagonal and the centre of $\cl T_L(\bb H^+)$.

\begin{corollary}\label{diag}
The diagonal and the centre of $\cl T_L(\bb H^+)$ both coincide with $\cl C$.
\end{corollary}
\begin{proof}
The maps $\rho_{s,t}$ are automorphisms of  $\cl T_L(\bb H^+)$ and hence leave
its centre $\cl Z$ invariant. By Proposition \ref{tree} (1),
if $A\in \cl Z$ then $\Phi_{k,m}(A)\in\cl Z$. By Corollary \ref{five} (b),
$L_{f_{k,m}(A)}L_u^kL_v^m\in\cl Z$ for each $k,m\ge 0$.
It is now immediate that if such an operator
commutes with all $L_u$ and $L_v$ then
$L_{f_{k,m}(A)}=0$ unless $k=m=0$. Thus, $A=L_{f_{0,0}(A)}\in\cl C$.

It follows from Proposition \ref{tree} (1) that $\Phi_{k,m}(A)^*=\Phi_{-k,-m}(A^*)$.
Hence by Corollary \ref{five} (a), if $A$ and $A^*$ are both in $\cl T_L(\bb H^+)$,
then $\Phi_{k,m}(A)=0$ unless $k = m = 0$.
Thus, each $A_r$ is in $\cl C$ and hence so is $A$.

We have shown that the centre and the diagonal are contained in $\cl C$.
The opposite inclusions are obvious.
\end{proof}


In some of the results that follow we adapt techniques used by Davidson and Pitts in \cite{davpit}.
Along with the left regular representation $L$ of $\bb{H}^+$ defined above, we
consider the restriction of its right regular representation to
$\cl H = \ell^2(\bb H^+)$. This is generated by the operators
\begin{align*}
R_u(w^n\otimes u^k\otimes v^m) = &w^{n-m}\otimes u^{k+1}\otimes v^m \\
R_v(w^n\otimes u^k\otimes v^m) = &w^{n}\otimes u^k\otimes v^{m+1}\\
R_w(w^n\otimes u^k\otimes v^m) = &w^{n+1}\otimes u^{k}\otimes v^m,
\qquad (n,k,m)\in \bb Z\times\bb Z_+ \times\bb Z_+.
\end{align*}
We denote by $\cl T_R(\bb H^+)$ the weak-* closed subalgebra of $\cl B(\ell^2(\bb H^+))$ generated by
$$\{ R_w^n,R_u^k, R_v^m: (n,k,m)\in \bb Z\times\bb Z_+ \times\bb Z_+\}.$$
It is trivial to verify that $\cl T_L(\bb H^+)$ and $\cl T_R(\bb H^+)$ commute.

\begin{lemma}\label{cyc}
Suppose that the operator $A \in\bh$ commutes with $\cl T_R(\bb H^+)$
and that $A(w^{0}\ot u^0\ot v^0)=0$. Then $A=0$.
\end{lemma}
\begin{proof} For each $(n,k,m)\in \bb{Z}\times\bb{Z}_+\times\bb{Z}_+$ we have
\begin{align*}
A(w^n\ot u^k\ot v^m) &= AR_{v^m}R_{u^k}R_{w^n}(w^{0}\ot u^0\ot v^0)\\
&= R_{v^m}R_{u^k}R_{w^n}A(w^{0}\ot u^0\ot v^0)=0.
\end{align*}
Hence, $A = 0$.
\end{proof}

The argument below is standard; for the case of the unilateral shift, see \cite[Prop. V.1.1]{davcstar}.
We include a proof for the convenience of the reader.

\begin{proposition}\label{essen}
If $A \in \bh$ commutes with $R_u$ or $R_v$, then  $\|A\|$ equals the essential norm
$\|A\|_{e}\equiv\inf\{\nor{A+K}:K \;\text{compact}\}$.
In particular, the algebra $\cl T_L(\bb H^+)$ does not contain nonzero compact operators.
\end{proposition}

\begin{proof} Assume that $A$ commutes with $R_v$ (the other case is similar).
It is easy to see that $(R_v^n)_n$ tends to $0$ weakly. Indeed if $x,y$ are in $\cl H$ and we write
$x=\sum_mx_m\ot v^m, \,  y=\sum_my_m\ot v^m$ where $x_m,y_m$ are in $L^2(\bb T)\ot \ell^2(\bb Z_+)$, then
\[
 \sca{R_v^nx,y}= \sum_m\sca{x_m,y_{m+n}}\to 0
\]
since $(\nor{x_m})$ and  $(\nor{y_m})$ are square integrable.

Suppose, by way of contradiction, that there is a  compact operator $K\in\bh$ such that $\|A+K\|<\|A\|$.
Then there is a unit vector
$x\in\cl H$ which satisfies  $\|Ax\|> \|A+K\| $.
But $\|(A+K)R_v^nx\|\leq \|A+K\|$ since $R_v^n$ is an isometry.
On the other hand, since $R_v^n$ tends to $0$ weakly we have $\lim_n\|KR_v^nx\|=0$. Thus
$$\lim_n \|(A+K)R_v^nx\|=\lim_n \| A R_v^nx\|=\lim_n \|R_v^nAx\|=\|Ax\|,$$
a contradiction.
\end{proof}

\begin{theorem}\label{quasi}
The algebra $\cl T_L(\bb H^+)$ does not contain quasinilpotent operators. In particular,
$\cl T_L(\bb H^+)$ is semi-simple.
\end{theorem}
\proof
Let $A \in\cl T_L(\bb H^+)$ be nonzero and define $f_{k,m}=f_{k,m}(A)\in L^\infty(\bb T)$
as in Corollary \ref{five}. 
Recall that for $r\in (0,1)$ we have set
\[
A_r=\sum_{k,m\ge 0}r^{k+m}L_{f_{k,m}}L_u^kL_v^m.
\]
Let \begin{align*}
     E= & \{(k,m): f_{k,m}\neq 0 \}, \qquad \rho=\inf \{k+m: (k,m) \in E\}\\
k_0= &\inf \{k: (k,m)\in E, k+m=\rho\}, \quad m_0=\rho-k_0.
    \end{align*}
If $g,h\in L^2(\bb T)$ and $n\in \bb Z_+$, we have
\begin{align*}
 & \sca{A_r^n(g\ot u^0\ot v^0), ( h\ot u^{nk_0}\ot v^{nm_0})} =\\
&\sum_\gamma
r^{\sum k_i}r^{\sum m_i}\!
\sca{(f_{k_1,m_1}\dots f_{k_n, m_n})\phi_\gamma g)\ot u^{\sum k_i}\ot v^{\sum m_i}\!, (h\ot u^{nk_0}\ot v^{nm_0})}
\end{align*}
where the summation is over all $\gamma = ((k_1,m_1),(k_2,m_2),...,(k_n, m_n))$ with $(k_i,m_i)\in E$
and $\phi_\gamma$ is a function of modulus $1$ such that
$L_u^{k_1}L_v^{m_1}\dots L_u^{k_n}L_v^{m_n}$ $=$ $L_{\phi_{\gamma}} L_u^{\sum k_i}L_v^{\sum m_i} $.
For a term in the above sum to be nonzero, we must have $\sum k_i=nk_0$ and $\sum m_i=nm_0$.
Thus, since
$k_i + m_i \ge\rho= k_0 + m_0$ for each $i$ and
 $\sum (k_i+m_i)=n(k_0+m_0)$ we obtain  $k_i + m_i=k_0 + m_0$ for all $i=1,\dots, n$.
But $k_i\geq k_0$ for all $i$, hence the condition $\sum k_i = nk_0$ gives
$k_i=k_0$ for all $i$ and so $m_i=m_0$ for all $i$.

Hence  there is only one nonzero term in the above sum and we obtain
\begin{align*}
\sca{A_r^n(g\ot u^0\ot v^0), ( h\ot u^{nk_0}\ot v^{nm_0})}
= r^{n(k_0+m_0)}\sca{(f_{k_0,m_0}^n\phi_{\gamma_0} g), h},
\end{align*}
where $\gamma_0 = ((k_0,m_0),(k_0,m_0),\dots,(k_0, m_0))$ (and the term $(k_0,m_0)$
appears $n$ times).
Now since $\nor{A_r}\le \nor{A}$ for each $r$ and $A_r\to A$ in the weak-* topology,
\begin{align*}
& \left|\sca{A^n(g\ot u^0\ot v^0), ( h\ot u^{nk_0}\ot v^{nm_0})}\right|  \\
&=\lim_{r\nearrow 1}\left|\sca{A_r^n(g\ot u^0\ot v^0), ( h\ot u^{nk_0}\ot v^{nm_0})}\right| \\
&= \lim_{r\nearrow
1}r^{n(k_0+m_0)}|\sca{(f_{k_0,m_0}^n\phi_{\gamma_0} g), h}|
=|\sca{(f_{k_0,m_0}^n \phi_{\gamma_0} g), h}| .
 \end{align*}
Since $\phi_{\gamma_0}$ is unimodular,
\begin{align*}
\nor{A^n} &\ge
\sup\left\{\left|\sca{A^n(g\ot u^0\ot v^0), h\ot u^{nk_0}\ot v^{nm_0}}\right|:\nor{g}_2\le 1, \nor{h}_2\le 1\right\} \\
& = \sup\{|\sca{f_{k_0,m_0}^n\phi_{\gamma_0} g, h}|:\nor{g}_2\le 1,
\nor{h}_2\le 1\}= \nor{f_{k_0,m_0}^n}_\infty.
 \end{align*}
Thus,
\[
\|A^{n}\|^{1/n}\geq \|f_{k_{0},m_{0}}\|_\infty
\]
for all $n$, and hence the spectral radius of $A$ is non-zero.
\endproof

\begin{theorem}\label{comm}
The commutant of  $\cl T_R(\bb H^+)$ is  $\cl T_L(\bb H^+)$.
\end{theorem}
\proof
Let $A$ be in the commutant of $\cl T_R(\bb H^+)$. Then 
$$ A(w^0\ot u^0\ot v^0)=\sum_{k,m\ge 0}\phi_{k,m}\ot u^k\ot v^m  $$
for some $\phi_{k,m}\in L^2(\bb T)$.

We show that $\phi_{k,m} \in L^{\infty}(\mathbb{T})$.
Let $g\in L^\infty(\bb T)$. Since $L_gA = AL_g$ (note that $L_g\in\cl Z\subseteq\cl T_R(\bb H^+))$, we have
\begin{align*}
A(g\ot u^0\ot v^0) &=L_gA(w^0\ot u^0\ot v^0)=\sum_{k,m\ge 0}L_g(\phi_{k,m}\ot u^k\ot v^m)\\
&= \sum_{k,m\ge 0}(g\phi_{k,m}\ot u^k\ot v^m)
\end{align*}
and so
$$\sca{A(g\ot u^0\ot v^0), (g\ot u^k\ot v^m)}=\sca{g\phi_{k,m},g}
=\frac{1}{2\pi}\int_0^{2\pi} \phi_{k,m}(t)|g(t)|^{2}dt.$$
Therefore
$$\left|\frac{1}{2\pi}\int_0^{2\pi}\phi_{k,m}(t)|g(t)|^{2}dt\right| \leq \|A\|\|g\|_2^2.$$
Using this inequality for characteristic functions in the place of $g$,
one sees that $\phi_{k,m}$ induces a linear functional on $L^1(\bb{T})$
of norm not larger than $\nor{A}$; thus, $\phi_{k,m}\in L^{\infty}(\bb{T})$.

We show that if $r\in (0,1)$ the operator
$A_r=\sum_{k,m\in \bb Z} \Phi_{k,m}(A)  r^{|k|+|m|}$ defined in Proposition \ref{tree}
is in the commutant of $\cl T_R(\bb H^+)$.
It suffices to show that $\Phi_{k,m}(A)=\sum_{i,j} Q_{k+i,m+j}AQ_{i,j}$
is in the commutant of $\cl T_R(\bb H^+)$
for all $k,m\in\bb Z$. We have
$R_{u}Q_{k,m}=Q_{k+1,m}R_{u}$
and hence
\begin{align*}
\sum_{i,j}  Q_{k+i,m+j}AQ_{i,j}R_{u}& =
\sum_{i,j} Q_{k+i,m+j}AR_{u}Q_{i-1,j} = \sum_{i,j} Q_{k+i,m+j}R_{u}AQ_{i-1,j}\\
& = R_{u}\sum_{i,j} Q_{k-1+i,m+j}AQ_{i-1,j}.
\end{align*}
Similarly,
$R_{v}Q_{k,m}=Q_{k,m+1}R_{v}$
and hence
\begin{align*}
\sum_{i,j}  Q_{k+i,m+j}AQ_{i,j}R_{v}& =
\sum_{i,j} Q_{k+i,m+j}AR_{v}Q_{i,j-1} = \sum_{i,j} Q_{k+i,m+j}R_{v}AQ_{i,j-1}\\
& =  R_{v}\sum_{i,j}Q_{k+i,m-1+j}AQ_{i,j-1}.
\end{align*}
\noindent Now set $B_r=\sum_{k,m\ge 0} r^{k+m}L_{\phi_{k,m}}L_{u^{k}}L_{v^{m}}$. Since $\phi_{k,m}\in L^\infty(\bb T)$,
the series converges absolutely to an operator in $\cl T_L(\bb H^+)$.

Clearly, $\Phi_{k,m}(A)(w^0\otimes u^0\otimes v^0) = \phi_{k,m}\otimes u^k\otimes v^m$
and so $A_r(w^0\ot u^0\ot v^0) = B_{r}(w^0\ot u^0\ot v^0)$.
Since both $A_r$ and $B_r$ are in the commutant of $\cl T_R(\bb H^+)$, Lemma \ref{cyc}
implies that $A_r=B_r$. Hence $A_r\in \cl T_L(\bb H^+)$. Since
$\cl T_L(\bb H^+)$ is weak-* closed,
Proposition \ref{tree} (2) implies that $A \in \cl T_L(\bb H^+)$.
\endproof

The following properties of $\cl T_L(\bb H^+)$ follow from Theorem \ref{comm}.

\begin{corollary} (a) The algebra $\cl T_L(\bb H^+)$ has the bicommutant property \\
$\cl T_L(\bb H^+)''=\cl T_L(\bb H^+)$.

(b) $\cl T_L(\bb H^+)$ is an inverse closed algebra.

(c) $\cl T_L(\bb H^+)$ is closed in the weak operator topology.
\end{corollary}

\section{Reflexivity of $\cl T_L(\bb{H}^+)$}\label{s_rllr}

In this section we establish the reflexivity of the algebra $\cl T_L(\bb{H}^+)$.
Let $F : L^2(\bb{T})\otimes L^2(\bb{T})\otimes L^2(\bb{T})\rightarrow
\ell^2(\bb{Z})\otimes \ell^2(\bb{Z})\otimes \ell^2(\bb{Z})$
be the tensor product of three copies of the Fourier transform.
Let $\cl K = H^2(\bb T)\otimes H^2(\bb T)$ and
$\widetilde{\cl H}= L^2(\bb{T})\otimes \cl K = L^2(\bb{T},\cl K)$;
we have that $\widetilde{\cl H} = F^{-1}(\ell^2(\bb{H}^+))$.
We will use the same symbol for the restriction of $F$ to $\widetilde{\cl H}$.

Let $\widetilde{W} = F^{-1}L_w F$, $\widetilde{U} = F^{-1}L_{u} F$, $\widetilde{V} = F^{-1}L_{v} F$
(acting on $\widetilde{\cl H})$ and
$\cl L=F^{-1}\cl T_L(\bb{H}^+)F$.
For a fixed $\xi\in \bb{T}$, let
$V_{\xi} = A_{\xi}\otimes S\in \cl B(H^2\otimes H^2)$, where $S=T_{\zeta_1}$ is the shift on $H^2$
and $A_{\xi}$ is given by
$(A_{\xi}f)(z) = f\left(\frac{z}{\xi}\right)$, $f\in H^2$.

Write $\mu$ for normalised Lebesgue measure on $\bb{T}$.
We consider the Hilbert space $\widetilde{\cl H}$ as a direct integral over the
measure space $(\bb{T},\mu)$ of the constant field
$\xi\rightarrow \cl K(\xi)=\cl K$ of Hilbert spaces.
Thus, an operator $T$ is decomposable \cite{afg} with respect to this field
if and only if it belongs to $\cl M\otimes\cl B(\cl K)$,
where $\cl M$ denotes the multiplication masa of $L^{\infty}(\bb{T})$;
we write $T=\int_\bb{T} T(\xi)d\mu(\xi)$.
We note that $\widetilde{W}$, $\widetilde{U}$ and $\widetilde{V}$
are decomposable.
In the next proposition we identify their direct integrals.

\begin{proposition}\label{diring}
When $\widetilde{\cl H}$ is identified with the direct integral over $(\bb{T},\mu)$ of the constant field
$\xi\rightarrow \cl K$
of Hilbert spaces, we have $\widetilde{W} = \int_{\bb{T}}\xi (I\otimes I)d\mu(\xi),$
$\widetilde{U} = \int_{\bb{T}} (S\otimes I) d\mu(\xi)$
and
$\widetilde{V} = \int_{\bb{T}} V_{\xi}d\mu(\xi)$.
\end{proposition}
\proof
We identify the elements of $\widetilde{\cl H} = L^2(\bb{T}, \cl K)$
with functions on three
variables, $f = f(\xi,z_1,z_2)$, such that for a.e.
$\xi\in\bb{T}$, the function on two variables
$f(\xi,\cdot,\cdot)$ is analytic.
To show that
$\widetilde{W} = \int_{\bb{T}} \xi (I\otimes I) d\mu(\xi)$, note that if $f\in
\widetilde{\cl H}$ then $\widetilde{W}f(\xi,z_1,z_2) = \xi f(\xi,z_1,z_2)$,
$\xi,z_1,z_2\in \bb{T}$.

The claim concerning $\widetilde{U}$ is immediate from its
definition. For $\widetilde{V}$ we argue as follows: let
$f(\xi,z_1,z_2) = \xi^n z_1^k z_2^m$
(that is, $f =F^{-1}(w^n\ot u^k\ot v^m))$; then
\begin{align*}
\widetilde{V}f &=\widetilde{V}F^{-1}(w^n\ot u^k\ot v^m) = F^{-1}L_v(w^n\ot u^k\ot v^m) \\
&= F^{-1}(w^{n-k}\ot u^k\ot v^{m+1})
 \end{align*} and thus
$\widetilde{V}f(\xi,z_1,z_2) = \xi^{n - k}z_1^kz_2^{m+1}$.
On the other hand, the direct integral
$\int_{\bb{T}}(A_{\xi}\otimes I) d\mu(\xi)$ transforms the
function $f$ into the function $g(\xi,z_1,z_2) =
\xi^{n-k}z_1^k z_2^m$. We thus have that
$$\widetilde{V} =
(I\otimes S)\int_{\bb{T}}(A_{\xi}\otimes I) d\mu(\xi)
= \int_{\bb{T}} (A_{\xi}\otimes S) d\mu(\xi).$$
\endproof

For $\xi\in \bb{T}$, let $\cl L_{\xi}\subseteq \cl B(\cl K)$
be the weak-* closed subalgebra
generated by $S\ot I$ and $V_{\xi}$.
The operators $A_\xi,S\in \cl B(H^2)$ are easily seen to satisfy
the assumptions of Corollary \ref{hase} with $\lambda = \bar\xi$.
It follows that $\cl L_\xi$ is reflexive; in particular, it is weakly closed.
We note that the algebra $\cl L_{\xi}$ was studied by Hasegawa in \cite{ha}
where a class of invariant subspaces of $\cl L_{\xi}$ was exhibited.

In the next theorem, we use the notion of
a direct integral of non-selfadjoint operator algebras developed in \cite{afg}.

\begin{theorem}\label{drinfol}
The algebra $\cl T_L(\bb{H}^+)$ is reflexive.
\end{theorem}
\proof
By definition, $\cl L= F^{-1} \cl T_L(\bb{H}^+)F$ is generated,
as a weak-* closed algebra, by 
the operators $\widetilde{U}$, $\widetilde{V}$, $\widetilde{W}$ and $\widetilde{W}^{-1}$.

Note that $\cl L\subseteq \cl M\otimes\cl B(\cl K)$; moreover,
$\cl L$ is  weakly closed, since $\cl T_L(\bb H^+)$ is a commutant (Theorem \ref{comm}).
Hence, by \cite{afg}, $\cl L$
gives rise to a direct integral $\int_{\bb{T}} \cl A(\xi) d\mu(\xi)$, where
$\cl A(\xi)$ is the weakly closed algebra generated by
$\widetilde{U}(\xi)$, $\widetilde{V}(\xi)$, $\widetilde{W}(\xi)$ and $\widetilde{W}^{-1}(\xi)$.
Since the operators $\widetilde{W}(\xi)$ and $\widetilde{W}^{-1}(\xi)$ are scalar multiples
of the identity, we have that $\cl A(\xi) = \cl L_\xi$.
 On the other hand, since $\cl M\otimes I_{\cl K}\subseteq \cl L$,
 all diagonal operators of the integral decomposition
are contained in $\cl L$. Proposition 3.3 of \cite{afg} shows that an operator $T=\int_\bb{T} T(\xi)d\mu(\xi)$
belongs to $\cl L$ if and only if
almost all $T(\xi)$ belong to $\cl L_\xi$.
As observed above, $\cl L_{\xi}$ is reflexive for each $\xi\in \bb{T}$.
Proposition 3.2 of \cite{afg} now implies that $\cl L$ is reflexive.
Therefore so is $\cl T_L(\bb{H}^+)$.
\endproof

\section{Other representations}

Until now, we were concerned with the left regular representation of
the Heisenberg semigroup. In this section, we consider another class
of representations defined as follows. Let $\gl=e^{2\pi i\theta}$
with $\theta$ irrational and $\ga : \bb{T}\rightarrow\bb{T}$ be the
rotation corresponding to $\theta$, that is, the map given by
$\ga(z)=\gl z$. We let $\nu$ be a  Borel probability measure on $\bb
T$ which is quasi-invariant (that is, $\nu(E) = 0$ implies
$\nu(\alpha(E)) = 0$, for every measurable set $E\subseteq \bb{T}$)
and ergodic (that is, $f\circ \alpha^k = f$ for all $k\in \bb{Z}$
implies that $f$ is constant, for every  $f\in L^{\infty}(\bb{T},\nu)$).
Let $\cl W_\pi(\bb H^+)$ be the
weak-* closed subalgebra of $\cl B(L^2(\bb T,\nu))$ generated by the
operators
\[
\pi(u) = M_{\zeta_1}, \qquad\pi(v)f=\sqrt{\frac{d\nu_{\gl}}{d\nu}}(f\circ\alpha) \quad\text{and }\;\pi(w)=\gl I
\]
where $M_{\zeta_1}$ is the operator of multiplication by the function $\zeta_1$ on $L^2(\bb{T},\nu)$
(recall that $\zeta_n(z)=z^n)$ and
$\nu_\lambda(A)$ is the Borel measure on $\bb{T}$ given by $\nu_\lambda(A)=\nu(\alpha(A))$.

We will need the following two lemmas;
the results are probably known in some form, but we have been unable to locate a precise reference and so
we include their proofs.
Below, the terms \emph{singular} and \emph{absolutely continuous}
are understood with respect to Lebesgue measure $\mu$.

\begin{lemma}\label{meas}
(i) The measure $\nu$ is either absolutely continuous or singular.

(ii) If $\nu$ is absolutely continuous it is equivalent to Lebesgue measure.

(iii) If $\nu$ is singular and not continuous, it is supported on an orbit of $\alpha$.
\end{lemma}
\proof (i) Denote by $\nu_{a}$ (resp. $\nu_{s}$) the absolutely
continuous (resp. singular) part of $\nu$. Suppose that $\nu_{s}\neq
0$ and $\nu_{a}\neq 0$ and let $A$ be a Borel set of Lebesgue
measure zero such that $\nu_{s}(\mathbb{T}\setminus A) = 0$. Then
$\cup_{n \in \mathbb{Z}} \alpha^{n}(A)$ is an invariant set of
positive $\nu$-measure. On the other hand, the Lebesgue measure of
$\cup_{n \in \mathbb{Z}} \alpha^{n}(A)$ is zero and hence $\cup_{n
\in \mathbb{Z}} \alpha^{n}(A)$ is not of full $\nu$-measure.
This contradicts the ergodicity of $\nu$.

\medskip

(ii) The set on which the Radon-Nikodym derivative $d\nu/d\mu$
vanishes is invariant by quasi-invariance of $\nu$; hence it is $\mu$-null by ergodicity of $\mu$.

\medskip

(iii) Let $z_0 \in \mathbb{T}$ be such that $\nu(\{z_0\}) \neq 0$. Then the orbit
$X= \{\ga^n(z_0): n \in \mathbb{Z}\}$ of $z_0$ is an invariant set of positive $\nu$-measure
and it follows from ergodicity that its complement is $\nu$-null.
\endproof

\begin{lemma}\label{p_fmr}
Let $\nu$ be a singular continuous measure. Then the weak-* closed hull of the linear span
of the set $\{M_{\zeta_n}: \  n=1,2...\}$ is equal to
$\{M_{f}: \ f \in L^{\infty}(\mathbb{T}, \nu)\}$.
\end{lemma}
\proof Let $f \in L^{1}(\mathbb{T}, \nu)$ be such that $$\int
f\zeta_n d\nu=0$$ for all $n=1,2,\dots$. It follows from the F. and M.
Riesz Theorem that the measure $fd\nu$ is absolutely continuous. Since $\nu$ is singular,
we obtain that $f=0$ $\nu$ a.e., and hence it is equal to $0$ as an element of  $L^{1}(\mathbb{T}, \nu)$.
\endproof

The next theorem completely describes the operator algebras arising
from the class of representations that we consider.

\begin{theorem}\label{refl}
Let $\cl N = \{\zeta_k H^2 : k\in\bb{Z}\}$.
\begin{enumerate}
\item
If $\nu$ is equivalent to Lebesgue measure, then the algebra
$\cl W_\pi(\bb H^+)$ is unitarily equivalent to the nest algebra
$\Alg\cl N$.
\item
If $\nu$ is singular and not continuous, then
$\cl W_\pi(\bb H^+)$ is again unitarily equivalent to $\Alg\cl N$.
\item
If $\nu$ is singular and continuous, then
$\cl W_\pi(\bb H^+) = \cl B(L^2(\bb T, \nu))$.
\end{enumerate}
\end{theorem}
\proof (1) Since $\nu$ is equivalent to Lebesgue measure, we may
assume that $\cl W_\pi(\bb H^+)$ acts on $L^2(\bb T)$,
$\pi(u)=M_{\zeta_1}$ and $\pi(v)f=f\circ\ga$.

If $a = (a_n)_{n\in\bb{Z}}\in
l^{\infty}(\bb{Z})$, let $D_a$ be given by $(\widehat{D_a f})(n)=
a_n \widehat{f}(n)$; thus $D_a$ is the image, under conjugation by
the Fourier transform, of the diagonal operator on $l^2(\bb{Z})$
given by $(x_j)\to (a_jx_j)$. Let $\cl D = \{D_a : a\in
\ell^{\infty}(\bb{Z})\}$; clearly, $\cl D$ is a masa on
$L^2(\mathbb{T})$. Since the map $\gs \to D_{(\gs^n)_n}$
is weak-* continuous from $\bb T$ into $\cl B(L^2(\bb{T}))$ and
$\{\gl^k :k\in \bb Z_+\}$ is dense in $\bb T$, the weak-* closed linear
span of
$\{D_{(\lambda^{kn})_n}:k\in \bb Z_+\} = \{\pi(v)^k :k\in\bb{Z}_+\}$
contains $\{ D_{(\gs^n)_n}: \gs\in\bb T\}$, hence
is a selfadjoint algebra and by the Bicommutant Theorem equals
$\cl D$. On the other hand, if $a\in \ell^\infty(\bb Z)$ and $p\geq 0$, the matrix
of $\pi(u)^pD_a$ with respect to the basis
$\{\zeta_k\}_{k\in\bb{Z}}$ has the sequence $a$ at the $p$-th
diagonal and zeros elsewhere. It follows that all lower triangular
matrix units belong to the algebra $\cl W_\pi(\bb H^+)$, and hence
it equals $\Alg\cl N$.

\medskip

(2) By lemma \ref{meas} (iii), $\nu$ is supported on the orbit
of a point $z_0\in \bb{T}$.
For $k\in\bb Z$, write $z_k=\ga^{-k}(z_0)$ and $\beta_k^2=\nu(\{z_k\})$. Since
$\nu_\gl(\{z_k\})= \nu(\{\ga(z_k)\})=\nu(\{z_{k-1}\})$ we have
$\beta_{k-1}=\beta(z_k)\beta_k$ where $\beta$ is the function determined by the
identity $\beta^2 = \dfrac{d\nu_\gl}{d\nu}$. If
$f_k=\dfrac{\chi_{\{z_k\}}}{\beta_k}$, then $\{f_k:k\in\bb Z\}$ is an
orthonormal basis of $L^2(\bb T, \nu)$ and we have
$\pi(v)\chi_{\{z_k\}}=\beta\cdot(\chi_{\{z_k\}}\circ\ga)=\beta\chi_{\{z_{k+1}\}}$.
Thus,
\[
\pi(v)f_k=\beta\frac{\chi_{\{z_{k+1}\}}}{\beta_k}=  \frac{\beta_k}{\beta_{k+1}}\frac{\chi_{\{z_{k+1}\}}}{\beta_k}=f_{k+1},
\]
and so $\pi(v)$ is the bilateral shift with respect to $\{f_k\}$.
Also $\pi(u)f_k=z_kf_k=\bar\gl^kz_0f_k$ for each $k$ and hence, as in the proof of (1),
the linear span of the positive powers of $\pi(u)$ is weak-* dense
in the set of all operators diagonalized by $\{f_k\}$. It follows as in (1) that
$\cl W_\pi(\bb H^+)$ consists of all operators which are lower
triangular with respect to $\{f_k\}$, hence it is unitarily equivalent to   $\Alg\cl N$.

\medskip

(3) By Lemma \ref{p_fmr}, the algebra $\cl W_\pi(\bb H^+)$
contains a maximal abelian selfadjoint algebra, namely, the
multiplication masa of $L^{\infty}(\bb{T},\nu)$.
Since $\ga$ acts ergodically, it is standard
that $\cl W_\pi(\bb H^+)$ has no nontrivial invariant subspaces. It follows
from \cite{arv100} that it is weak-* dense in, and hence equal to, $\cl B(L^2(\bb{T},\nu))$.
\endproof

\begin{remark}
Note the different roles of $\pi(u)$ and $\pi(v)$ in  (1) and (2):
in (1), the diagonal masa is generated by (the nonnegative powers of) $\pi(v)$;
in (2) the masa is generated by $\pi(u)$. These two representations
generate  inequivalent representations of the irrational rotation algebra,
as the corresponding measures are not equivalent (see \cite{br}).
\end{remark}

\bigskip

\noindent\textbf{A non-reflexive representation } We now construct
an example of a representation of $\bb{H}^+$ which generates a
non-reflexive weakly closed operator algebra. This representation,
$\rho$, acts on $H^2$ and is defined as follows:
If $S = T_{\zeta_1}$ is the shift and $V\in \cl B(H^2)$ is the operator given by
$(Vf)(z)=f(\gl z) = (f\circ\alpha)(z)$, we define
\[
\rho(u) = S, \qquad\rho(v) = SV \quad\text{and }\;\rho(w) = \gl I
\]
with $\gl=e^{2\pi i \theta}$ and $\theta$ irrational.
Let $\cl A$ be the weakly closed algebra generated by $\rho(u)$ and
$\rho(v)$. Using Fourier transform, we identify $H^2$ with
$\ell^2(\bb{N})$ and let $E : \cl B(H^2)\to \cl D\simeq
\ell^\infty(\bb N)$ be the usual normal conditional expectation onto
the diagonal given by $E((a_{ij})) = (b_{ij})$ where
$b_{ij}=a_{ij}\gd_{ij}$. Define $E_k$ for $k\ge 0$ by  $E_k(A) =
E((S^{*})^{k }A)$.

We recall that $[\cl S]$ denotes the linear span of a subset $\cl S$
of a vector space.

\begin{proposition}\label{thin}
If $A\in \cl A$ then $E_ m(A)\in
[I,V,\dots,V^ m]$.
\end{proposition}
\begin{proof}
The operator $A$ is
the weak limit of polynomials of the form \\ $\sum_{k,n\geq 0}
c_{k,n}S^{k+n}V^n$.
Thus, $E_m(A)$ is a weak limit of polynomials of the form
$$\sum c_{ k ,n} V^n$$
where the summation is over all $k,n\in\bb Z_+$ with $k +n = m$ and hence
$E_ m(A)\in [I,V,\dots,V^ m]$.
\end{proof}

\begin{proposition}\label{l_co}
If $\cl K\in\Lat\{S,SV\}$ then in fact $\cl K\in\Lat\{S,V\}$ and hence $\cl K=\zeta_kH^2$ for
some $k\in\bb Z_+$.
\end{proposition}
\begin{proof}

Since $S(\cl K)\subseteq \cl K$ and $\cl K\subseteq H^2$, by Beurling's Theorem there is an inner function
$\phi$ such that $\cl K=\phi H^2$.
Since $SV(\cl K)\subseteq \cl K$,
we have $SV(\phi)\in \cl K=\phi H^2$, so $\frac{z\phi(\gl z)}{\phi(z)}\in H^\infty$.
Thus,  there exists $h\in H^\infty$ such that
\begin{equation}\label{equ}
z\phi(\gl z)=h(z)\phi(z)\quad\text{for all }\; z\in\bb D.
\end{equation}
Let $\phi_{1}$ be an analytic function and $l$ a non-negative integer such that
$\phi_{1}(0)\neq 0$  and $\phi(z)=z^{l}\phi_{1}(z)$ for all $z \in \bb{D}$.
We obtain
\begin{equation}\label{equII}
z^{l+1}\lambda^{l}\phi_{1}(\lambda z)=h(z)z^{l}\phi_{1}(z)\quad\text{for all }\; z\in\bb D
\end{equation}
and  hence
\begin{equation}\label{equIII}
z\lambda^{l}\phi_{1}(\lambda z)=h(z)\phi_{1}(z)\quad\text{for all }\; z\in\bb D.
\end{equation}

Setting $z=0$ in (\ref{equIII}), we obtain that $h(0) = 0$. Thus,
there exists $h_1\in H^\infty$ such that $h(z) = z h_1(z)$. The relation
$z \phi(\gl z) = h(z)\phi(z) = z h_1(z)\phi(z)$ implies $\phi\circ\alpha=h_1\phi$
and hence $(\phi\circ\alpha) H^2 \subseteq  \phi H^2$. Therefore
$$V(\cl K)=V(\phi H^2)=(\phi\circ\alpha) H^2 \subseteq  \phi H^2=\cl K.$$
Considering $\cl K$ as a subspace of $L^2(\bb T)$, Theorem \ref{refl} (1) gives that $\cl K=\zeta_k H^2$ for some $k$
(note that here $\nu$ equals Lebesgue measure); since $\cl K\subseteq H^2$, $k$ must be nonnegative.
\end{proof}

\begin{theorem}
The algebra $\cl A$ is not reflexive; in fact
$\Ref\cl A = \Alg\cl N$ where $\cl N=\{\zeta_k H^2: k\in\bb Z_+\}$.

\end{theorem}
\begin{proof}
By Proposition \ref{l_co}, $\Ref\cl A = \Alg\cl N$.
It follows from Proposition \ref{thin} that $\cl A$ is strictly contained in $\Ref\cl A$.
\end{proof}

\noindent{\bf Acknowlegment} We wish to thank Vassili Nestoridis for useful discussions.

\end{document}